\documentclass[12pt]{amsart}
\title[Degrees of symmetric Grothendieck polynomials]{Degrees of symmetric Grothendieck polynomials and Castelnuovo-Mumford regularity}

\author{Jenna Rajchgot }
\address[JR]{
Dept.~of Mathematics and Statistics, 
McMaster University, 
Hamilton, ON  \newline \indent  L8S 4K1, CANADA}
\email{rajchgoj@mcmaster.ca}
\thanks{Jenna Rajchgot was partially supported by NSERC Grant RGPIN-2017-05732.}

\author{Yi Ren}
\address[YR]{
Physical and Theoretical Chemistry Laboratory,
University of Oxford,
Oxford \newline \indent OX1 3QZ, UK}
\email{yi.ren@chem.ox.ac.uk}
\thanks{Yi Ren was supported by NSERC Grant RGPIN--2017-05732.}

\author{Colleen Robichaux}
\address[CR]{
Dept.~of Mathematics,
University of Illinois at Urbana-Champaign,
Urbana, IL \newline \indent 61801, USA}
\email{cer2@illinois.edu}
\thanks{Colleen Robichaux was supported by the National Science Foundation Graduate Research Fellowship Program under Grant No. DGE 1746047.}

\author{Avery St.~Dizier}
\address[AS]{
Dept.~of Mathematics,
Cornell University,
Ithaca, NY 14853, USA}
\email{ajs624@cornell.edu}
\thanks{}

\author{Anna Weigandt}
\address[AW]{
Dept.~of Mathematics,
University of Michigan,
Ann Arbor, MI 48109, USA}
\email{weigandt@umich.edu}
\thanks{}

\date{\today}

\usepackage[bookmarks=false]{hyperref}

\usepackage[top=1in, left=1in, right=1in, bottom=1in]{geometry}
\usepackage{amsfonts, amssymb, latexsym, graphicx, amsmath,enumitem,amsthm}
\usepackage{color}
\usepackage{tikz}
\usepackage[vcentermath]{youngtab}
\usepackage{ytableau}
\ytableausetup{centertableaux}
\usepackage{algorithm,algorithmic}
\usetikzlibrary{calc, shapes, backgrounds,arrows,positioning,decorations.pathmorphing}

\newtheorem{theorem}{Theorem}[section]
\newtheorem{lemma}[theorem]{Lemma}
\newtheorem{conjecture}[theorem]{Conjecture}
\newtheorem{corollary}[theorem]{Corollary}
\newtheorem{proposition}[theorem]{Proposition}

\theoremstyle{definition}
\newtheorem{definition}[theorem]{Definition}
\newtheorem{example}[theorem]{Example}

\newcommand{\cellsize}{19}
\newlength{\cellsz} 
\newsavebox{\cell}
\sbox{\cell}{\begin{picture}(\cellsize,\cellsize)
	\put(0,0){\line(1,0){\cellsize}}
	\put(0,0){\line(0,1){\cellsize}}
	\put(\cellsize,0){\line(0,1){\cellsize}}
	\put(0,\cellsize){\line(1,0){\cellsize}}
	\end{picture}}
\newcommand\cellify[1]{\def\thearg{#1}\def\nothing{}%
	\ifx\thearg\nothing
	\vrule width0pt height\cellsz depth0pt\else
	\hbox to 0pt{\usebox{\cell} \hss}\fi%
	\vbox to \cellsz{
		\vss
		\hbox to \cellsz{\hss$#1$\hss}
		\vss}}
\newcommand\tableau[1]{\vtop{\let\\\cr
		\baselineskip -16000pt \lineskiplimit 16000pt \lineskip 0pt
		\ialign{&\cellify{##}\cr#1\crcr}}}

\newcommand\mydef[1]{{\bf #1}}

\newcommand\sylv[1]{{\tt sv}(#1)}

\newcommand{\gTab}{{\sf Tab}}

\newcommand \partitions[1]{\mathcal P_{#1}}

\begin{document}

\maketitle

\begin{abstract}We give an explicit formula for the degree of the Grothendieck polynomial of a Grassmannian permutation and a closely related formula for the Castelnuovo-Mumford regularity of the Schubert determinantal ideal of a Grassmannian permutation. 
We then provide a counterexample to a conjecture of Kummini-Lakshmibai-Sastry-Seshadri on a formula for regularities of standard open patches of particular Grassmannian Schubert varieties and show that our work gives rise to an alternate explicit formula in these cases. We end with a new conjecture on the regularities of standard open patches of \emph{arbitrary} Grassmannian Schubert varieties.  
\end{abstract}

\section{Introduction}\label{sec:intro}
	
	Lascoux and Sch\"utzenberger \cite{LS82} introduced  \emph{Grothendieck polynomials} to study the K-theory of flag varieties. 
	Grothendieck polynomials have a recursive definition, using divided difference operators. The symmetric group $S_n$ acts on the polynomial ring $\mathbb Z[x_1,x_2,\ldots,x_n]$  by permuting indices. Let $s_i$ be the simple transposition in $S_n$ exchanging $i$ and $i+1$. Then define operators on $\mathbb Z[x_1,x_2,\ldots,x_n]$ 
 \[\partial_i=\frac{1-s_i}{x_i-x_{i+1}} \text{ and } \pi_i=\partial_i(1-x_{i+1}).\]     
	Write  $w_0=n \, n-1 \, \ldots \, 1$ for the \mydef{longest permutation} in $S_n$ (in one-line notation) and take 
   \[\mathfrak G_{w_0}(x_1,x_2,\ldots,x_n)=x_1^{n-1}x_2^{n-2}\cdots x_{n-1}.\]
	Let $w_i:=w(i)$ for $i\in[n]$. Then if $w_i>w_{i+1}$, we define $\mathfrak G_{s_iw}=\pi_i(\mathfrak G_w)$.  We call $\{\mathfrak G_w:w\in S_n\}$ the set of \mydef{Grothendieck polynomials}.  Since the $\pi_i$'s satisfy the same braid and commutation relations as the simple transpositions, each $\mathfrak G_w$ is well defined. 
	
Grothendieck polynomials are generally inhomogeneous.  The lowest degree of the terms in $\mathfrak G_w$  is given by the \emph{Coxeter length} of $w$. The degree (i.e.\ highest degree of the terms) of $\mathfrak G_w$ can be described combinatorially in terms of pipe dreams (see \cite{FominKrillov,KM.Subword}), but this description is not readily computable. We seek an explicit combinatorial formula. 
In this paper, we give such an expression in the Grassmannian case.  Our proof relies on a formula of Lenart \cite{Le00}.

	One motivation for wanting easily-computable formulas for degrees of Grothendieck polynomials (for large classes of $w\in S_n$) comes from commutative algebra: formulas for degrees of Grothendieck polynomials give rise to closely related formulas for   \emph{Castelnuovo-Mumford regularity} of associated Schubert determinantal ideals. Recall that Castelnuovo-Mumford regularity is an invariant of a homogeneous ideal related to its minimal free resolution (see Section~\ref{sect:reg} for definitions).  Formulas for regularities of Schubert determinantal ideals yield formulas for regularities of certain well-known classes of generalized determinantal ideals in commutative algebra. 
For example, among the Schubert determinantal ideals are ideals of $r\times r$ minors of an $n\times m$ matrix of indeterminates and one sided ladder determinantal ideals.  
Furthermore, many other well-known classes of generalized determinantal ideals can be viewed as defining ideals of Schubert varieties intersected with opposite Schubert cells, so degrees of specializations of \emph{double Grothendieck polynomials} govern Castelnuovo-Mumford regularities in these cases. Thus, one purpose of this paper is to suggest a purely combinatorial approach to studying regularities of certain classes of generalized determinantal ideals. 
	
	\section{Background on Permutations}
	We start by recalling some background on the symmetric group.  We follow \cite{Manivel} as a reference.  Let $S_n$ denote the \mydef{symmetric group} on $n$ letters, i.e.\ the set of bijections from the set $[n]:=\{1,2,\ldots, n\}$ to itself.  We typically represent permutations in one-line notation.  The \textbf{permutation matrix} of $w$, also denoted by $w$, is the matrix which has a $1$ at $(i,w_i)$ for all $i\in[n]$, and zeros elsewhere. 
	
	The \mydef{Rothe diagram} of $w$ is the subset of cells in the $n\times n$ grid
	\[ D(w)= \{(i, j) \  \mid \ 1 \leq i, j \leq  n,\  w_i > j, \text{ and } w^{-1}_j > i\}.\]
	Graphically, $D(w)$ is the set of cells in the grid which remain after plotting the points $(i,w_i)$ for each $i\in[n]$ and striking out any boxes which appear weakly below or weakly to the right of these points.
The \mydef{essential set} of $w$ is the subset of the diagram
\[\mathcal{E}ss(w)=\{(i,j)\in D(w) \ \mid \ (i+1,j),(i,j+1)\not \in D(w)\}.\]  
	Each permutation has an associated \mydef{rank function} defined by \[r_w(i,j)=|\{(i',w_{i'}) \ \mid  \ i'\leq i,w_{i'} \leq j \}|.\]
 We write $\ell(w):=|D(w)|$ for the {\bf Coxeter length} of $w$.

\begin{example}\label{ex:Rothe}
If 
$w=63284175 \in {S}_{8}$ (in one-line notation)
then $D(w)$ is the following: 
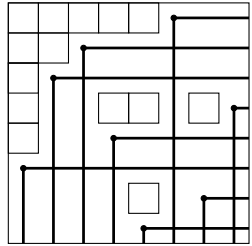
\[
\begin{tikzpicture}[scale=.4]
\draw (0,0) rectangle (8,8);

\draw (0,7) rectangle (1,8);
\draw (1,7) rectangle (2,8);
\draw (2,7) rectangle (3,8);
\draw (3,7) rectangle (4,8);
\draw (4,7) rectangle (5,8);

\draw (0,6) rectangle (1,7);
\draw (1,6) rectangle (2,7);

\draw (0,5) rectangle (1,6);
\draw (0,3) rectangle (1,4);

\draw (0,4) rectangle (1,6);
\draw (4,4) rectangle (5,5);
\draw (3,4) rectangle (4,5);

\draw (6,4) rectangle (7,5);

\draw (4,1) rectangle (5,2);

\filldraw (5.5,7.5) circle (.5ex);
\draw[line width = .2ex] (5.5,0) -- (5.5,7.5) -- (8,7.5);
\filldraw (2.5,6.5) circle (.5ex);
\draw[line width = .2ex] (2.5,0) -- (2.5,6.5) -- (8,6.5);
\filldraw (1.5,5.5) circle (.5ex);
\draw[line width = .2ex] (1.5,0) -- (1.5,5.5) -- (8,5.5);
\filldraw (7.5,4.5) circle (.5ex);
\draw[line width = .2ex] (7.5,0) -- (7.5,4.5) -- (8,4.5);
\filldraw (3.5,3.5) circle (.5ex);
\draw[line width = .2ex] (3.5,0) -- (3.5,3.5) -- (8,3.5);
\filldraw (0.5,2.5) circle (.5ex);
\draw[line width = .2ex] (0.5,0) -- (0.5,2.5) -- (8,2.5);
\filldraw (6.5,1.5) circle (.5ex);
\draw[line width = .2ex] (6.5,0) -- (6.5,1.5) -- (8,1.5);
\filldraw (4.5,0.5) circle (.5ex);
\draw[line width = .2ex] (4.5,0) -- (4.5,0.5) -- (8,0.5);
\end{tikzpicture}
\]
Here ${\mathcal E}ss(w)=\{(1,5), (2,2), (4,5),(4,7), (5,1), (7,5)\}$.
\end{example}

	\section{Grassmannian Grothendieck Polynomials}
	A \mydef{partition} is a weakly decreasing sequence of nonnegative integers $\lambda=(\lambda_1,\lambda_2,\ldots,\lambda_k)$. We define the \mydef{length} of $\lambda$ to be $\ell(\lambda)=|\{h\in[k] \mid \lambda_h\neq 0\}|$ and the \mydef{size} of $\lambda$, denoted $|\lambda|$, to be $\sum_{i=1}^k\lambda_i$. 
  Write $\partitions{k}$ for the set of partitions of length at most $k$.
	Here, we conflate partitions with their Young diagrams, i.e.\ the notation $(i,j)\in \lambda$ indicates choosing the $j$th box in the $i$th row of the Young diagram of $\lambda$.  
  
  We say $w\in S_n$ has a \mydef{descent} at position $k$ if $w_k>w_{k+1}$.
	A permutation $w\in S_n$ is \mydef{Grassmannian} if $w$ has a unique descent.
	To each Grassmannian permutation $w$, we can uniquely associate a partition $\lambda\in \partitions{k}$: \[\lambda=(w_k-k,\ldots,w_1-1),\] where $k$ is the position of the descent of $w$.

  Let $w_{\lambda}$ denote the Grassmannian permutation associated to $\lambda$. It is easy to check that
	\begin{equation}\label{eq:sameLen}
	    |\lambda|=\ell(w_{\lambda})=|D(w_{\lambda})|.
	\end{equation} 

Define ${\sf YTab}(\lambda)$ to be the set of fillings of $\lambda$ with entries in $[k]$ so that
\begin{itemize}
\item entries weakly increase from left-to-right along rows and
\item entries strictly increase from top-to-bottom along columns.
\end{itemize}
For a partition $\lambda$, the \mydef{Schur polynomial} in  $k$ variables is
\[s_{\lambda}(x_1,x_2,\ldots,x_k)=\sum_{T\in {\sf YTab(\lambda)}}\prod_{i=1}^k x_i^{\# i\text{ 's in }T}.\]

	\begin{definition}
		Let $\lambda,\mu\in \partitions{k}$ so that $\lambda\subseteq \mu$. Denote by $\gTab(\mu/\lambda)$ the set of fillings of the skew shape $\mu/\lambda$ with entries in $[k]$ such that 
		\begin{itemize}
			\item  entries strictly increase left-to-right in each row,
		    \item entries strictly increase top-to-bottom in each column, and
			\item entries in row $i$ are at most $i-1$ for each $i\in[k]$.
		\end{itemize}
	\end{definition}

For ease of notation, let $\mathfrak G_{{\lambda}}:=\mathfrak G_{w_{\lambda}}$.

	\begin{theorem}\cite[Theorem~2.2]{Le00}\label{thm:gGrass}
		For a Grassmannian permutation $w_{\lambda}\in S_n$,
		\begin{align*}\label{equation:gGrass}
		\mathfrak G_{{\lambda}}(x_1,x_2,\ldots,x_k)=\sum_{\substack{\mu\in \partitions{k}\\ \lambda\subseteq\mu}}a_{\lambda\mu}s_{\mu}(x_1,x_2,\ldots,x_k)
		\end{align*}
		where $(-1)^{|\mu|-|\lambda|}a_{\lambda,\mu}=| \gTab(\mu/\lambda)|$ and $k$ is the unique descent of $w_\lambda$.
	\end{theorem}

	\begin{example}
		The Grassmannian permutation $w=2 4 8 1 3 5 6 7$ corresponds to $\lambda=(5,2,1)$. 
		By Theorem~\ref{thm:gGrass},
		\[\mathfrak G_{(5,2,1)}(x_1,x_2,x_3)=s_{(5,2,1)}-2s_{(5,2,2)}-s_{(5,3,1)}+2s_{(5,3,2)}-s_{(5,3,3)}.\]
		This corresponds to the tableaux:
		\[
		\ytableausetup
		{boxsize=0.9em}
		\begin{ytableau}
		*(lightgray)  & *(lightgray)  & *(lightgray) & *(lightgray) & *(lightgray)  \\
		*(lightgray)  & *(lightgray) \\
		*(lightgray)     
		\end{ytableau}
		\quad
		\begin{ytableau}
		*(lightgray)  & *(lightgray)  & *(lightgray) & *(lightgray) & *(lightgray)  \\
		*(lightgray)  & *(lightgray) \\
		*(lightgray)  & \scriptstyle1   
		\end{ytableau}
		\quad
		\begin{ytableau}
		*(lightgray)  & *(lightgray)  & *(lightgray) & *(lightgray) & *(lightgray)  \\
		*(lightgray)  & *(lightgray) \\
		*(lightgray)  & \scriptstyle2 
		\end{ytableau}
		\quad
		\begin{ytableau}
		*(lightgray)  & *(lightgray)  & *(lightgray) & *(lightgray) & *(lightgray) \\
		*(lightgray)  & *(lightgray) & \scriptstyle1\\
		*(lightgray)     
		\end{ytableau}
		\quad
		\begin{ytableau}
		*(lightgray)  & *(lightgray)  & *(lightgray) & *(lightgray) & *(lightgray) \\
		*(lightgray)  & *(lightgray) & \scriptstyle1\\
		*(lightgray)   & \scriptstyle1  
		\end{ytableau}
		\quad
		\begin{ytableau}
		*(lightgray)  & *(lightgray)  & *(lightgray) & *(lightgray) & *(lightgray) \\
		*(lightgray)  & *(lightgray) & \scriptstyle1\\
		*(lightgray)   & \scriptstyle2  
		\end{ytableau}
		\quad
		\begin{ytableau}
		*(lightgray)  & *(lightgray)  & *(lightgray) & *(lightgray) & *(lightgray) \\
		*(lightgray)  & *(lightgray) & \scriptstyle1\\
		*(lightgray)   & \scriptstyle 1 & \scriptstyle2  
		\end{ytableau}
		\]
	\end{example}

	\medskip
	
	\begin{definition}
		We say a partition $\mu$ is \mydef{maximal} for $\lambda$ if $\gTab(\mu/\lambda)\neq\emptyset$ and  $\gTab(\nu/\lambda)=\emptyset$ whenever $|\nu|>|\mu|$.
	\end{definition}

The following lemma can be obtained from the proof of \cite[Theorem~2.2]{Le00}, but we include it for completeness.
	\begin{lemma}\label{lemma:maxLambdaRec}
		Fix a partition $\lambda\in {\mathcal P}_k$. Define $\mu$ by setting $\mu_1=\lambda_1$, and $\mu_i=\min\{\mu_{i-1},\lambda_{i}+(i-1)\}$ for each $1<i\leq k$. Then $\mu$ is the unique partition that is maximal for $\lambda$.
	\end{lemma}
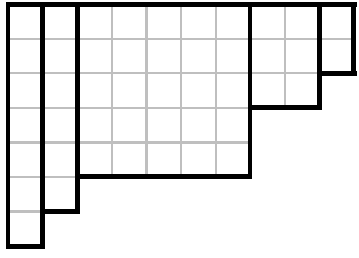
	\begin{proof}
		Let $\rho$ be any partition with $\gTab(\rho/\lambda)\neq\emptyset$. Since elements of $\gTab(\rho/\lambda)$ have strictly increasing rows, $\rho/\lambda$ has at most $i-1$ boxes in row $i$ for each $i$. That is, $\rho_i\leq \lambda_i+(i-1)$ for each $i$. It follows that $\rho_i\leq \mu_i$ for each $i$. Thus, uniqueness of $\mu$ will follow once we show that $\mu$ is maximal for $\lambda$. It suffices to produce an element $T\in\gTab(\mu/\lambda)$.
		
		We will denote by $T(i,j)$ the filling by $T$ of the box in row $i$ and column $j$ of $\mu$. For each $i$ and $j$ with $\lambda_i<j\leq\mu_i$, set 
		\[T(i,j)=i+j-\mu_i-1.\] 
		It is easily seen that $T$ strictly increases along rows with $T(i,j)\in[i-1]$ for each $i$. To see that $T\in\gTab(\mu/\lambda)$, it remains to note that $T$ strictly increases down columns. Observe
		\[
		T(i,j)-T(i-1,j)  
		= \mu_{i-1}-\mu_i+1
		>0. \qedhere
		\]
	\end{proof}

	\begin{example}\label{ex:sylvDecompFill}
		If $\lambda=(10,10,9,7,7,2,1)$, the unique partition $\mu$ maximal for $\lambda$ is $\mu=(10,10,10,10,10,7,7)$. Below is the tableau $T\in \gTab(\mu/\lambda)$ constructed in the proof of Lemma~\ref{lemma:maxLambdaRec}.
		\[\ytableausetup
		{boxsize=1em}
		\ytableausetup{notabloids}
		\begin{ytableau}
		*(lightgray)  &  *(lightgray) & *(lightgray)  & *(lightgray)  &  *(lightgray)& *(lightgray) & *(lightgray)& *(lightgray) & *(lightgray)& *(lightgray) \\
		*(lightgray)  &  *(lightgray) & *(lightgray)  & *(lightgray)  &  *(lightgray)& *(lightgray) & *(lightgray)& *(lightgray) & *(lightgray)& *(lightgray) \\
		*(lightgray)  &  *(lightgray) & *(lightgray)  & *(lightgray)  &  *(lightgray)& *(lightgray) & *(lightgray)& *(lightgray) & *(lightgray) & {2} \\
		*(lightgray)  & *(lightgray)  &  *(lightgray) & *(lightgray)  & *(lightgray)  &  *(lightgray)& *(lightgray)& {1} & {2} & {3}\\
		*(lightgray)  & *(lightgray)  &  *(lightgray) & *(lightgray)  & *(lightgray)  &  *(lightgray)& *(lightgray)& {2} & {3}& {4} \\
		*(lightgray)  &  *(lightgray)  & {1} & {2} & {3} & {4} & {5} \\
		*(lightgray)   & {1} & {2} & {3} & {4} & {5}& {6}
		\end{ytableau}
		\]
	\end{example}
	
	\medskip

\begin{definition}
		Given a partition $\lambda=(\lambda_1,\ldots,\lambda_k)$,    	let $P(\lambda)=(P_1,P_2,\ldots,P_{r})$  be the set partition of $[k]$ such that $i,j\in P_h$ if and only if
		$\lambda_i=\lambda_j$, and $\lambda_i>\lambda_j$ whenever $i\in P_{h}$ and $j\in P_{l}$ with $h<l$. 
	\end{definition}
	Note that if $\lambda=(\lambda_1,\ldots,\lambda_k)=(\lambda_{i_1}^{p_1},\ldots,\lambda_{i_r}^{p_r})$	in exponential notation, then $p_h=|P_h|$ for each $h\in[r]$.
	In the following definition, we describe a decomposition of $\lambda$ into rectangles.
	\begin{definition}
		Let $\lambda=(\lambda_1,\ldots,\lambda_k)$ be a partition and $P(\lambda)=(P_1,P_2,\ldots,P_{r})$. Set $m_h=\min P_h$ for each $h$. Define $R(\lambda)=(R_1,R_2,\ldots,R_{r})$ by setting
		\[R_h:=\left\{(i,j)\in \lambda \ \mid \ i\in \bigcup_{l=1}^h P_{l}\mbox{ and } \lambda_{m_{h+1}}<j\leq \lambda_{m_{h}}\right\},\]
		where we take $\lambda_{m_{r+1}}:=0$.
	\end{definition}

  Set $\lambda^{(h)}$ to be the partition
	\[\lambda^{(h)}=\bigcup_{j=1}^{h}R_j\] 
	for $h\in[r]$. Equivalently, for $h\in[r-1]$, $\lambda^{(h)}=(\lambda_1-\lambda_i,\lambda_2-\lambda_i,\ldots,\lambda_{i-1}-\lambda_i)$ where $i=\min P_{h+1}$, and $\lambda^{(r)}=\lambda$. Set $\lambda^{(0)}:=\emptyset$.
  
	\begin{example}\label{ex:lambdaToRect}
		For $\lambda$ as in Example~\ref{ex:sylvDecompFill}, one has  $P_1=\{1,2\}$, $P_2=\{3\}$, $P_3=\{4,5\}$, $P_4=\{6\}$, and $P_5=\{7\}$. The sets in $R(\lambda)$ are outlined below, with $R_1$ the rightmost rectangle and $R_5$ the leftmost. Considering $h=2$, $\lambda^{(h)}=R_1\cup R_2=(10-7,10-7,9-7)=(3,3,2)$.
		
		\setlength{\unitlength}{.23mm}
		\[
		\begin{picture}(200,140)

		\thinlines
		\put(0,120){{\color{lightgray}\line(1,0){200}}}
		\put(0,100){{\color{lightgray}\line(1,0){200}}}
		\put(0,80){{\color{lightgray}\line(1,0){180}}}
		\put(0,60){{\color{lightgray}\line(1,0){140}}}
		\put(0,40){{\color{lightgray}\line(1,0){40}}}
		\put(0,20){{\color{lightgray}\line(1,0){40}}}
		\put(0,0){\color{lightgray}{\line(1,0){20}}}
		\put(60,40){{\color{lightgray}\line(0,1){100}}}
		\put(80,40){\color{lightgray}{\line(0,1){100}}}
		\put(100,40){\color{lightgray}{\line(0,1){100}}}
		\put(120,40){\color{lightgray}{\line(0,1){100}}}
		\put(160,80){\color{lightgray}{\line(0,1){60}}}
		
		\linethickness{0.5mm}
		\put(-1,140){\line(1,0){202}}
		\put(180,100){\line(1,0){21}}
		\put(140,80){\line(1,0){41}}
		\put(40,40){\line(1,0){101}}
		\put(20,20){\line(1,0){21}}
		\put(-1,0){\line(1,0){22}}
		\put(0,0){\line(0,1){141}}
		\put(20,0){\line(0,1){140}}
		\put(40,20){\line(0,1){120}}
		\put(140,40){\line(0,1){100}}
		\put(180,80){\line(0,1){60}}
		\put(200,100){\line(0,1){40}}
		
		\end{picture}
		\]
	\end{example}

	\begin{definition}
		For any $n\geq1$, let $\delta^{n}$ denote the \mydef{staircase shape} $\delta^{n}=(n,n-1,\ldots,1)$. Given a partition $\mu$, let
		\[\sylv{\mu}=\max\left\{k  \mid \delta^{k}\subseteq \mu\right\}.\] 
		The partition $\delta^{\sylv{\mu}}$ is called the \mydef{Sylvester triangle} of $\mu$.
	\end{definition}

	\begin{proposition}\label{prop:muTriang}
		Suppose $\mu$ is maximal for $\lambda$ and $P(\lambda)=(P_1,\ldots,P_r)$. If $i\in P_{h+1}$ for some $0\leq h<r$, then 
		\[\mu_i=\lambda_i+\sylv{\lambda^{(h)}}.\]
	\end{proposition}

	\begin{proof}
		By Lemma~\ref{lemma:maxLambdaRec}, $\mu_1=\lambda_1$ and $\mu_i=\min\{\mu_{i-1},\lambda_{i}+(i-1)\}$ for $1<i\leq k$. Clearly $P(\lambda)$ refines $P(\mu)$: if $\lambda_i=\lambda_j$, then $\mu_i=\mu_j$. Example~\ref{ex:sylvDecompFill} shows this refinement can be strict. Hence, it suffices to prove the statement when $i=\min P_{h+1}$. We work by induction on $h$.
		
		When $h=0$, $i=\min(P_1)=1$. Since $\lambda_1=\mu_1$, the result follows. Suppose the claim holds for some $h-1$. We show the claim holds for $h$. Let $i=\min P_{h+1}$. Then it suffices to show that 
		\begin{equation}\label{eq:equalDef}
		\lambda_i+\sylv{\lambda^{(h)}}
		=\min\{\mu_{i-1},\lambda_{i}+(i-1)\}.
		\end{equation}
		Since $i=\min P_{h+1}$, it follows that $i-1\in P_{h}$. By applying the inductive assumption to $\mu_{i-1}$,
		\begin{equation}\label{eq:useInd}
		\min\{\mu_{i-1},\lambda_{i}+(i-1)\}=\min\{\lambda_{i-1}+\sylv{\lambda^{(h-1)}},\lambda_{i}+(i-1)\}.
		\end{equation}
		By Equations (\ref{eq:equalDef}) and (\ref{eq:useInd}), the proof is complete once we show
		\begin{equation}\label{eq:sylRecLam}
		\sylv{\lambda^{(h)}}=\min\{(\lambda_{i-1}-\lambda_{i})+\sylv{\lambda^{(h-1)}},i-1\}.
		\end{equation}
		
		Let $\omega, \ell$ respectively denote the (horizontal) width and (vertical) length of $R_{h}$, and set $M=\sylv{\lambda^{(h-1)}}$. Equation (\ref{eq:sylRecLam}) is equivalent to proving
		\[\sylv{\lambda^{(h)}}=\min\{\omega+M,\ell\}.\]
		By definition, 
		$\lambda^{(h)}=R_{h}\cup \lambda^{(h-1)}$, so it is straightforward to see that \[\sylv{\lambda^{(h)}}\leq\min\{\omega+M,\ell\}.\]
		Let $(M,c)$ be the southwest most box in the northwest most embedding of 
		$\delta^{M}\subseteq \lambda^{(h-1)}$, with the indexing inherited from $\lambda$.
		
		Suppose first that $\ell\geq \omega+M$. Since $R_{h}$ is a rectangle, $(\omega+M,c-\omega)\in \lambda^{(h)}$.
		Then $\delta^{\omega+M}\subseteq \lambda^{(h+1)}$ and Equation (\ref{eq:sylRecLam}) follows. Otherwise, it must be that $\ell-M< \omega$. Since $R_{h}$ is a rectangle, $(\ell,c-\ell+M)\in \lambda^{(h)}$. Thus, $\delta^\ell\subseteq \lambda^{(h+1)}$ and Equation (\ref{eq:sylRecLam}) follows.
	\end{proof}

	\begin{theorem}\label{thm:Grass}
		Suppose $w_{\lambda}\in S_n$ is a Grassmannian permutation. Let $P(\lambda)=(P_1,\ldots,P_{r})$. Then
		\[{\rm deg}(\mathfrak G_\lambda)=|\lambda|+\sum_{h\in[r-1]}|P_{h+1}|\cdot  \sylv{\lambda^{(h)}}.\]
	\end{theorem}

	\begin{proof}
		By Theorem~\ref{thm:gGrass} and Lemma~\ref{lemma:maxLambdaRec}, the highest nonzero homogeneous component of $\mathfrak{G}_\lambda$ is $a_{\lambda\mu}s_{\mu}$ where $\mu$ is maximal for $\lambda$. Since $\deg(s_{\mu})$ is $|\mu|$,
		Proposition~\ref{prop:muTriang} implies the theorem, using the fact that $\sylv{\lambda^{(0)}}=0$.
	\end{proof}

\begin{example}\label{ex:thmCalc}
Returning to $\lambda$ as in Example~\ref{ex:sylvDecompFill}, Theorem~\ref{thm:Grass} states that ${\rm deg}(\mathfrak G_\lambda)=
|\lambda|+\sum_{h=1}^4|P_{h+1}|\cdot \sylv{\lambda^{(h)}}=46+(1\cdot 1+ 2\cdot 3+1\cdot 5 +1\cdot 6)=46+18=64$.

\end{example}

\section{Castelnuovo-Mumford regularity of Grassmannian matrix Schubert varieties}\label{sect:reg}

In this section, we recall some basics of Castelnuovo-Mumford regularity and then use Theorem~\ref{thm:Grass} to produce easily-computable formulas for the regularities of matrix Schubert varieties associated to Grassmannian permutations.

\subsection{Commutative algebra preliminaries}
Let $S = \mathbb{C}[x_1,\dots, x_n]$ be a positively $\mathbb{Z}^d$-graded polynomial ring so that the only elements in degree zero are the constants. 
The \textbf{multigraded Hilbert series} of a finitely generated graded module $M$ over $S$ is
\[
H(M; \mathbf{t}) = \sum_{\mathbf{a}\in \mathbb{Z}^d}\text{dim}_{K}(M_{\mathbf{a}})\mathbf{t}^{\mathbf{a}}
=\frac{\mathcal{K}(M;\mathbf{t})}{\prod_{i=1}^n(1-\mathbf{t}^{\mathbf{a_i}})},~~~\text{deg}(x_i) = \mathbf{a_i},
\]
where if $\mathbf{a}_i = (a_i(1),\dots, a_i(d))$, then $\mathbf{t}^{\mathbf{a}_i} = t_1^{a_i(1)}\cdots t_d^{a_i(d)}$. 
The numerator $\mathcal{K}(M;\bf{t})$ in the expression above is a Laurent polynomial in the $t_i$'s, called the $\mathbf{K}$\textbf{-polynomial} of $M$.
For more detail on $K$-polynomials, see \cite[Chapter 8]{Miller.Sturmfels}.

We are mostly interested in the case where $S$ is standard graded, that is, $\text{deg}(x_i) = 1$, and the case where $M = S/I$ where $I$ is a homogeneous ideal with respect to the standard grading. Note that, in this case, the $K$-polynomial is a polynomial in a single variable $t$. 
There is a minimal free resolution 
\[
 0 \rightarrow \bigoplus_jS(-j)^{\beta_{l,j}(S/I)}\rightarrow \bigoplus_jS(-j)^{\beta_{l-1,j}(S/I)}\rightarrow \cdots \rightarrow \bigoplus_jS(-j)^{\beta_{0,j}(S/I)} \rightarrow S/I \rightarrow 0
\]
where $l\leq n$ and $S(-j)$ is the free $S$-module obtained by shifting the degrees of $S$ by $j$. 
The \mydef{Castelnuovo-Mumford regularity} of $S/I$, denoted $\text{reg}(S/I)$, is defined as
\[
\text{reg}(S/I):=\text{max}\{j-i \mid \beta_{i,j}(S/I)\neq 0\}.
\]
This invariant is measure of complexity of $S/I$ and has multiple homological characterizations. For example, $\text{reg}(S/I)$ is the least integer $m$ for which $\text{Ext}^j(S/I,S)_n = 0$, for all $j$ and all $n\leq -m-j-1$ (see \cite[Proposition~20.16]{EisenbudCA}). We refer the reader to \cite[Chapter 20.5]{EisenbudCA} for more information on regularity. 

Let $\mathcal{K}(S/I;t)$ denote the $K$-polynomial of $S/I$ with respect to the standard grading. 
Assume that $S/I$ is Cohen-Macaulay and let $\text{ht}_S I$ denote the height of the ideal $I$. Then,
\begin{equation}\label{eq:mainRegEquation}
\text{reg}(S/I) = \text{deg }\mathcal{K}(S/I;t) - \text{ht}_S I.
\end{equation}
See, for example, \cite[Lemma~2.5]{Benedetti.Varbaro} and surrounding explanation. 
In this paper, we will use this characterization of regularity.

\subsection{Regularity of Grassmannian matrix Schubert varieties}

Let $X$ be the space of $n\times n$ matrices with entries in $\mathbb{C}$, let $\widetilde{X}=(x_{ij})$ denote an $n\times n$ generic matrix of variables, and let $S = \mathbb{C}[x_{ij}]$. Given an $n\times n$ matrix $M$, let $M_{[i,j]}$ denote the submatrix of $M$ consisting of the top $i$ rows and left $j$ columns of $M$. Given a permutation matrix  $w\in S_n$ we have the \textbf{matrix Schubert variety}
\[
X_w := \{M\in X\mid \text{rank } M_{[i,j]}\leq \text{rank } w_{[i,j]}\},
\]
which is an affine subvariety of  $X$ with defining ideal
\[
I_w := \langle (r_w(i,j)+1)-\text{size minors of } \widetilde{X}_{[i,j]}\mid (i,j)\in \mathcal{E}ss(w)\rangle \subseteq S.
\]
The ideal $I_w$, called a \textbf{Schubert determinantal ideal}, is prime \cite{Fulton.Flags} and is homogeneous with respect to the standard grading of $S$. 

By \cite[Theorem~A]{KM}, we have ${\mathcal K}(S/I_w;t)=\mathfrak G_w(1-t,\ldots,1-t)$, which has the same degree as $\mathfrak G_w(x_1,\ldots,x_n)$, 
since the coefficients in homogeneous components of single Grothendieck polynomials have the same sign (see, for example, \cite{KM}). Thus,
\begin{equation}\label{eq:prelimReg}
\text{reg}(S/I_w) = \text{deg } \mathfrak{G}_w(x_1,\dots, x_n) - \text{ht}_S I_w = \text{deg } \mathfrak{G}_w(x_1,\dots, x_n) - |D(w)|,
\end{equation}
where the second equality follows because 
\[\text{ht}_S I_w=\text{codim}_X X_w=|D(w)|\]
by \cite{Fulton.Flags}.
We now turn our attention to the case where $w$ is a Grassmannian permutation and retain the notation from the previous section.

\begin{corollary}\label{cor:grReg}
Suppose $w_{\lambda}\in S_n$ is a Grassmannian permutation. Let $P(\lambda)=(P_1,\ldots,P_{r})$. Then
		\[{\rm reg}(S/I_{w_\lambda})=\sum_{h\in [r-1]}|P_{h+1}|\cdot  \sylv{\lambda^{(h)}}.\]
\end{corollary}

\begin{proof}
This is immediate from Theorem~\ref{thm:Grass}, Equation \eqref{eq:prelimReg}, and Equation \eqref{eq:sameLen}.
\end{proof}

\begin{example}
Continuing Example~\ref{ex:thmCalc}, Corollary~\ref{cor:grReg} states that ${\rm reg}(S/I_{w_\lambda})=18$.
\end{example}

\begin{example}\label{eg:det}
The ideal of $(r+1)\times (r+1)$ minors of a generic $n\times m$ matrix is the Schubert determinantal ideal of a Grassmannian permutation $w\in S_{n+m}$. Indeed, $w$ is the permutation of minimal length in $S_{n+m}$ such that $\text{rank }w_{[n,m]} = r$.

The corresponding partition is $\lambda=(m-r)^{(n-r)}0^r$.  We have $\lambda^{(1)}=(m-r)^{(n-r)}$ and so $\sylv{\lambda^{(1)}}=\min\{m-r,n-r\}$.  Furthermore, $|P_2|=r$.
Therefore, \[{\rm reg}(S/I_w)=r\cdot\min\{m-r,n-r\}=r\cdot(\min\{m,n\}-r).\]
\end{example}

We claim no originality for the formula in Example~\ref{eg:det}; minimal free resolutions of ideals of $r\times r$ minors of a generic $n\times m$ matrix are well-understood (see \cite{Lascoux78} or \cite[Chapter~6]{Weyman}). 

\section{On the regularity of coordinate rings of Grassmannian Schubert varieties intersected with the opposite big cell}

	In this section, we discuss a conjecture of Kummini-Lakshmibai-Sastry-Seshadri \cite{KLSS} on Castelnuovo-Mumford regularity of coordinate rings of certain open patches of Grassmannian Schubert varieties. We provide a counterexample to the conjecture, and then we state and prove an alternate explicit formula for these regularities. We end with a conjecture on regularities of coordinate rings of standard open patches of \emph{arbitrary} Schubert varieties in Grassmannians. 

\subsection{Grassmannian Schubert varieties in the opposite big cell}

Fix $k\in [n]$ and let $Y$ denote the space of $n\times n$ matrices of the form
\begin{equation}\label{eq:Y}
\begin{bmatrix} M & I_{k}\\ I_{n-k} &0\end{bmatrix},
\end{equation}
where $M$ is a $k\times (n-k)$ matrix with entries in $\mathbb{C}$ and $I_{k}$ is a $k\times k$ identity matrix. 
Let $P\subseteq GL_n(\mathbb{C})$ denote the maximal parabolic of block lower triangular matrices with block rows of size $k, (n-k)$ (listed from top to bottom). Then the  Grassmannian of $k$-planes in $n$-space, $Gr(k,n)$, is isomorphic to $P\backslash GL_n(\mathbb{C})$. Further, the map $\pi: GL_n(\mathbb{C})\rightarrow Gr(k,n)$ given by taking a matrix to its coset mod $P$ induces an isomorphism from $Y$ onto an affine open subvariety $U$ of $Gr(k,n)$ (often called the opposite big cell). 

Let $B\subseteq GL_n(\mathbb{C})$ be the Borel subgroup of upper triangular matrices. \textbf{Schubert varieties} $X_w$ in $P\backslash GL_n(\mathbb{C})$ are closures of orbits $P\backslash PwB$, where $w\in S_n$ is a Grassmannian permutation with descent at position $k$. Let $Y_w$ denote the affine subvariety of $Y$ defined to be $\pi|_Y^{-1}(X_w\cap U)$.

Let $\widetilde{Y}$ denote the matrix that has the form given in \eqref{eq:Y} with variables $m_{ij}$ as the entries of $M$. Then, the coordinate ring of $Y$ is $\mathbb{C}[Y] = \mathbb{C}[m_{ij}\mid i\in[k], j\in[n-k]]$, and the prime defining ideal $J_w$ of $Y_w$ is generated by the essential minors of $\widetilde{Y}$. 
That is, 
\begin{equation}
J_w = \langle (r_w(i,j)+1)-\text{size minors of }\widetilde{Y}_{[i,j]}\mid (i,j)\in \mathcal{E}ss(w)  \rangle.
\end{equation}

\subsection{A conjecture, counterexample, and correction}

We now consider a conjecture of Kummini-Lakshmibai-Sastry-Seshadri from \cite{KLSS} on regularities of coordinate rings of standard open patches of certain Schubert varieties in Grassmannians. We show that this conjecture is false by providing a counterexample, and then state and prove an alternate explicit combinatorial formula for these regularities. This latter result follows immediately from our Corollary~\ref{cor:grReg}.

To state the conjecture from \cite{KLSS}, we first translate the conventions from their paper to ours. 
Indeed, we use the same notation as the previous section and assume that $w\in S_n$ is a Grassmannian permutation with unique descent at position $k$. 
Suppose that $w = w_1~w_2\cdots w_n$ in one-line notation. Observe that $w$ is uniquely determined from $n$ and $(w_1,\dots, w_k)$. Suppose further that for some $r\in[k-1]$, 
\begin{equation}\label{eq:conjGrassCond}
    w_{k-r+i}=n-k+i \ \text{ for all }  i\in[r]
\end{equation}
and $w_1=1$.
Let $\widetilde{w}$ be defined by  $(\widetilde{w}_1,\dots,\widetilde{w}_k)= (n-w_k+1,\dots,n-w_1+1)$.  
Then we have 
\[(\widetilde{w}_1,\dots,\widetilde{w}_k) = (k-r+1,k-r+2,\dots,k, a_{r+1},\dots, a_{n-1}, n)\] 
for some $k<a_{r+1}<\cdots<a_{n-1}<n$.  Let $a_r = k$ and $a_k = n$. For $r \leq i \leq k-1$,
define $m_i = a_{i+1}-a_i$.

\begin{conjecture}[{\cite[Conjecture 7.5]{KLSS}}]
\label{conj:KLSS}
\begin{equation}\label{eq:conj}
\text{reg}(\mathbb{C}[Y]/J_w) = \sum_{i=r}^{k-1}(m_i-1)i.
\end{equation}
\end{conjecture}

\begin{example}
We consider \cite[Example 6.1]{KLSS}. 
Let $J$ be the ideal  generated by $3\times 3$ minors of a $4\times 3$ matrix of indeterminates. Then $J = J_w$ for $w = 1245367\in S_7$, where $k = 4$ and $n=7$. Then $\widetilde{w} = (3,4,6,7)$. Here we see that Equation (\ref{eq:conj}) yields a regularity of 2. This matches the regularity we computed in Example~\ref{eg:det}.  
\end{example}

We now show that Conjecture~\ref{conj:KLSS} is not always true. 

\begin{example}\label{eg:counter}
Let $k = 4$, $n=10$, $w = 145723689(10)$ so that $\widetilde{w} = (4,6,7,10)$. Then $\widetilde{w}$ has the desired form. Furthermore, we have that $m_1 = 2, m_2 = 1, m_3 = 3$. Thus, by Conjecture~\ref{conj:KLSS}, the regularity should be $(2-1)1+(1-1)2+(3-1)3 = 1+6 = 7$. However, a check in Macaulay2 \cite{M2} yields a regularity of $5$. In fact, $J_w$, once induced to a larger polynomial ring, is a Schubert determinantal ideal for $w$, so we can use our formula from Corollary~\ref{cor:grReg}.  Notice $w$ has associated partition $\lambda=(3,2,2,0)$.  Then $\lambda^{(1)}=(1)$ and $\lambda^{(2)}=(3,2,2)$, giving $\text{reg}(\mathbb{C}[Y]/J_w)=2\cdot\sylv{\lambda^{(1)}}+1\cdot\sylv{\lambda^{(2)}}=2\cdot 1+1\cdot3=5$.
\end{example}

As illustrated in Example~\ref{eg:counter}, our formula for the regularity of a Grassmannian matrix Schubert variety given in Corollary~\ref{cor:grReg} corrects Conjecture~\ref{conj:KLSS} whenever the ideal $J_w$ is equal (up to inducing the ideal to a larger ring) to the Schubert determinantal ideal $I_w$. In fact, each Grassmannian permutation considered in \cite[Conjecture 7.5]{KLSS} is of this form. This follows because all the essential set of such $w$ is contained in $w_{[k,n-k]}$ 
by Equation (\ref{eq:conjGrassCond}).

\begin{corollary}
Let $w_\lambda \in S_n$ be a Grassmannian permutation with descent at position $k$ such that $w_1=1$ and for some $r\in[k-1]$, $w_{k-r+i}=n-k+i$ for $i\in[r]$.
Let $P(\lambda)=(P_1,\ldots,P_{r})$. Then
		\[{\rm reg}(\mathbb{C}[Y]/J_{w_\lambda})=\sum_{h\in [r-1]}|P_{h+1}|\cdot  \sylv{\lambda^{(h)}}.\]
\end{corollary}

\subsection{A conjecture for the general case}
We end the paper with a conjecture for the regularity of $\mathbb{C}[Y]/J_w$ where $w$ is an \emph{arbitrary} Grassmannian permutation with descent at position $k$. We begin with some preliminaries. 

First note that $\mathbb{C}[Y]/J_w$ is a standard graded ring. Indeed, the torus $T\subseteq GL_n(\mathbb{C})$ of diagonal matrices acts on $U$ and on  $X_w\cap U$ by right multiplication. This action induces a $\mathbb{Z}^n$-grading on $\mathbb{C}[Y]$ such that $m_{ij}$ has degree $\vec{e}_i-\vec{e}_j$ and $J_w$ is homogeneous. 
This $\mathbb{Z}^n$-grading can be coarsened to the standard $\mathbb{Z}$-grading because the $T$-action contains the dilation action\footnote{More generally, coordinate rings of Kazhdan-Lusztig varieties $X_w\cap X^v_\circ \subseteq B_-\backslash GL_n(\mathbb{C})$ are standard graded when $v$, the permutation defining the opposite Schubert cell $X^v_\circ = B_-\backslash B_-vB_-$, is $321$-avoiding. See \cite[pg. 25]{Knutson-Frob} or \cite[Section~4.1]{WooYongGrobner} for further explanation.}: embed $\mathbb{C}^\times\hookrightarrow T$ by sending $z\in \mathbb{C}^\times$ to the diagonal matrix that has its $(i,i)$-entry equal to $1$ when $1\leq i\leq n-k$ and equal to $z$ when $n-k+1\leq i\leq n$. 

The codimension of $Y_w$ in $Y$ is equal to the number of boxes in the diagram $D(w)$. So, to compute the regularity $\text{reg}(\mathbb{C}[Y]/J_w)$, it remains to find the degree of the $K$-polynomial of $\mathbb{C}[Y]/J_w$. 
By \cite[Theorem~4.5]{WooYongGrobner}, this $K$-polynomial can be expressed in terms of a \mydef{double Grothendieck polynomial}, $\mathfrak G_w({\mathbf x};{\mathbf y})$, which is defined as follows: 
\[\mathfrak G_{w_0}({\mathbf x};{\mathbf y})=\prod_{i+j\leq n}(x_i+y_j-x_iy_j).\]
The rest are defined recursively, using the same operator $\pi_i$ and recurrence defined in Section~\ref{sec:intro}.
Note that if $G_w(\mathbf{x};\mathbf{y})$ denotes the double Grothendieck polynomials in \cite{KM}, we have $G_w(\mathbf{x};\mathbf{y})=\mathfrak G_w({\mathbf 1-\mathbf x};{\mathbf 1-\frac{\mathbf 1}{\mathbf y}})$.

Let $\mathbf{c} = ((1-t),(1-t),\dots, (1-t), 0,0, \dots, 0)$ be the list consisting of $k$ copies of $1-t$ followed by $n-k$ copies of $0$, and let $\mathbf{\tilde{c}} = (0,0,\dots,0, 1-\frac{1}{t}, 1-\frac{1}{t},\dots, 1-\frac{1}{t})$ be the list consisting of $n-k$ copies of $0$ followed by $k$ copies of $1-\frac{1}{t}$. By \cite[Theorem~4.5]{WooYongGrobner}, the $K$-polynomial of $S/J_w$ is the specialized double Grothendieck polynomial $\mathfrak{G}_w(\bf{c}; \tilde{\bf{c}})$\footnote{The conventions used in \cite{WooYongGrobner} differ from ours, so the given formula is a translation of their formula to our conventions.}. Consequently, we are reduced to computing the degree of this polynomial.

\begin{example}
Let $w=132$ and $k=2$.  Then 
\[
\mathfrak G_w(\mathbf x;\mathbf y)=(x_2+y_1-x_2y_1)+(x_1+y_2-x_1y_2)-(x_1+y_2-x_1y_2)(x_2+y_1-x_2y_1).
\]
Letting $\mathbf{c} = (1-t, 1-t, 0)$ and $\widetilde{\mathbf{c}} = (0, 1-\frac{1}{t}, 1-\frac{1}{t})$, one checks that $\mathfrak G_w(\mathbf{c};\widetilde{\mathbf{c}})=(1-t)$ which is the $K$-polynomial of $S/J_w$ with respect to the standard grading. 
\end{example}

For the reader familiar with pipe dreams (see, e.g.\ \cite{FominKrillov} and \cite{KM.Subword}), we note  that the degree of $\mathfrak{G}_w(\bf{c}; \tilde{\bf{c}})$ is the maximum number of plus tiles in a (possibly non-reduced) pipe dream for $w$ with all of its plus tiles supported within the northwest justified $k\times (n-k)$ subgrid of the $n\times n$ grid. This follows from \cite{WooYongGrobner}. However, this is not a very explicit formula for degree.

We now turn to our conjecture. It asserts that the degree of the $K$-polynomial of $\mathbb{C}[Y]/J_w$ for a Grassmannian permutation $w\in S_n$ with descent at position $k$ can be computed in terms of the degree of a Grothendieck polynomial of an associated \emph{vexillary} permutation. This will be a much more easily computable answer than a pipe dream formula because the first, third, and fifth authors will give an explicit formula for degrees of vexillary Grothendieck polynomials in the sequel.

A permutation $w\in S_n$ is \textbf{vexillary} if it contains no $2143$-pattern, i.e.\ there are no $i<j<k<l$ such that $w_j<w_i<w_l<w_k$. For example, $w=\underline{3}25\underline{1}\underline{6}\underline{4}$ is not vexillary since it contains the underlined the $2143$ pattern.

Suppose $w_\lambda\in S_n$ is Grassmannian with descent $k$. Define $\lambda'=(\lambda_1,\ldots,\lambda_{\ell(\lambda)})$ and $\phi(\lambda)=(\phi_1,\ldots,\phi_{\ell({\lambda})})$  as follows.
For $i\in [\ell({\lambda})]$, 
\begin{equation*}
    \phi_{i} =
    \begin{cases}
      i+\min\{(n-k)-\lambda_i,k-i\}       & \mbox{ if } \lambda_i>\lambda_{i+1} \mbox{ or } i=\ell({\lambda}),\\
      \phi_{i+1} & \mbox{ otherwise.}  \\
       
    \end{cases}
  \end{equation*}
A vexillary permutation $v$ is determined by the statistics of a partition and a flag, computed using $D(v)$ (see \cite[Proposition~2.2.10]{Manivel}). Thus, the partition $\lambda'$ and flag $\phi$ defined above from $w_{\lambda}$ define at most one vexillary permutation. 
  
\begin{conjecture}\label{conj:vex}
  Fix $w_{\lambda}\in S_n$ Grassmannian with descent $k$. 
Then $\lambda', \phi(\lambda)$ define a vexillary permutation $v$, and ${\rm deg}({\mathfrak G}_{w_\lambda}({\bf c}; \tilde{{\bf c}}))={\rm deg}({\mathfrak G}_v({\bf x}))$. In particular, $\text{reg}(\mathbb{C}[Y]/J_{w_\lambda})={\rm deg}({\mathfrak G}_v({\bf x}))-|\lambda|$.
\end{conjecture}

While we state this as a conjecture here, the first, third, and fifth authors will prove this in the sequel and furthermore give an explicit combinatorial formula for ${\rm deg}({\mathfrak G}_v({\bf x}))$, as mentioned above.

\begin{example}
Let $k=5,n=10$ and $w_\lambda=1489(10)23567$. Then $\lambda'=(5,5,5,2)$ and $\phi(w_\lambda)=(3,3,3,5)$, which corresponds to the vexillary permutation $v=678142359(10)$.
Thus Conjecture~\ref{conj:vex} states that ${\rm deg}({\mathfrak G}_{w_\lambda}({\bf c}; \tilde{{\bf c}}))={\rm deg}({\mathfrak G}_v({\mathbf x}))=18$, so $\text{reg}(\mathbb{C}[Y]/J_{w_\lambda})=18-17=1$.

To compute this regularity directly, take $R = \mathbb{C}[Y] = \mathbb{C}[m_{ij} \mid 1\leq i,j\leq 5]$. Let $G$ denote the set of $2\times 2$ minors of $\begin{bmatrix} m_{11} & m_{12} &m_{13}\\ m_{21}&m_{22}&m_{23}\end{bmatrix}$, and let $H$ be the set of entries in the bottom three rows of the matrix of variables $M = (m_{ij})_{1\leq i,j \leq 5}$. Then $G\cup H$ is a minimal generating set of $J_{w_\lambda}$. The Eagon-Northcott complex is a minimal free resolution of $R/\langle G\rangle$:
\[
0\rightarrow R(-3)^2\rightarrow R(-2)^3\rightarrow R\rightarrow R/\langle G\rangle \rightarrow 0.
\]
From this, one directly observes that the regularity of the $R$-module $R/\langle G\rangle$ is $1$. Modding out $R/\langle G\rangle$ by the linear forms in $H$ does not change the regularity (see, e.g.\ \cite[Proposition~20.20]{EisenbudCA}), and hence the regularity of $R/J_{w_\lambda}$ is also $1$. 
\end{example}

\section*{Acknowledgements}
The authors would like to thank Daniel Erman, Reuven Hodges, Patricia Klein, Claudiu Raicu, Alexander Yong, and the anonymous referee for their helpful comments and conversations.

\end{document}